\documentclass[12pt]{amsart}
\usepackage{version}
\usepackage{amssymb}

\parskip=\smallskipamount

\newtheorem{theorem}{Theorem}[section]

\newtheorem{lemma}[theorem]{Lemma}
\newtheorem{corollary}[theorem]{Corollary}
\newtheorem{proposition}[theorem]{Proposition}

\theoremstyle{definition}
\newtheorem{definition}[theorem]{Definition}
\newtheorem{example}[theorem]{Example}

\newtheorem{remark}[theorem]{Remark}

\newcommand{\B}{\mathbb{B}}
\newcommand{\C}{\mathbb{C}}
\newcommand{\D}{\mathbb{D}}

\newcommand{\N}{\mathbb{N}}

\newcommand{\R}{\mathbb{R}}

\newcommand{\id}{{\sf id}}

\def\v{\varphi}
\def\de{\partial}

\numberwithin{equation}{section}

\begin{document}
\title[embedding into Loewner chains]{Embedding  univalent functions in filtering Loewner chains in higher dimension}
\author[L. Arosio]{Leandro Arosio$^\dag$}
\address{L. Arosio: Dipartimento Di Matematica\\
Universit\`{a} di Roma \textquotedblleft Tor Vergata\textquotedblright\ \\
Via Della Ricerca Scientifica 1, 00133 \\
Roma, Italy} \email{arosio@mat.uniroma2.it}

\author[F. Bracci]{Filippo Bracci$^\dag$}
\address{F. Bracci: Dipartimento Di Matematica\\
Universit\`{a} di Roma \textquotedblleft Tor Vergata\textquotedblright\ \\
Via Della Ricerca Scientifica 1, 00133 \\
Roma, Italy} \email{fbracci@mat.uniroma2.it}

\author[E. F. Wold]{Erlend Forn\ae ss Wold$^{\dag\dag}$}
\address{E. F. Wold: Matematisk Institutt, Universitetet i Oslo,
Postboks 1053 Blindern, 0316 Oslo, Norway}
\email{erlendfw@math.uio.no}

%
%
\date{\today}

\thanks{$\dag$ Supported by the ERC grant ``HEVO - Holomorphic Evolution Equations'' n. 277691}
\thanks{$^{\dag\dag}$ Supported by the NFR grant 209751/F20}

\begin{abstract}
We discuss the problem of embedding univalent functions into Loewner chains in higher dimension. In particular, we prove that a normalized univalent map of the ball in $\C^n$ whose image is a smooth strongly pseudoconvex domain is embeddable into a normalized Loewner chain (satisfying also some extra regularity properties) if and only if the closure of the image is polynomially convex.
\end{abstract}

\maketitle

\section{Introduction}

Let $\B^n:=\{z\in \C^n: \|z\|^2<1\}$ denote the unit ball of $\C^n$. Let
\[
\mathcal S:=\{f: \B^n \to \C^n: f(0)=0, df_0={\sf id}, f \hbox{ univalent}\}.
\]
Recall that a {\sl normalized Loewner chain} $(f_t)$ is a family of univalent mappings $(f_t:\mathbb B^n\rightarrow\mathbb C^n)_{t\geq 0}$,  with   $\Omega_s:=f_s(\mathbb B^n)\subset f_t(\mathbb B^n)$ for $s\leq t$ and such that $f_t(0)=0$ and $df_0=e^t {\sf id}$ for all $t\geq 0$.
We set  $R(f_t):=\cup_{s\geq 0}\Omega_s\subseteq \mathbb C^n$ and call the \emph{Loewner range} of $(f_t)$ the class of biholomorphism of $R(f_t)$.
A normalized Loewner chain $(f_t)$ has always Loewner range biholomorphic to $\C^n$, but in dimension greater than $1$, due to the existence of Fatou-Bieberbach domains, the open set $R(f_t)$  might be strictly contained in $\C^n$ (see Section \ref{catene} for more details about Loewner chains).

We say that a function $f\in \mathcal S$ {\sl embeds} into a normalized Loewner chain if there exists a normalized Loewner chain $(f_t)$ such that $f_0=f$. We denote by
\[
\mathcal S^1:=\{f\in \mathcal S: f\ \hbox{embeds into a normalized Loewner chain $(f_t)$ with $R(f_t)=\C^n$}\}.
\]

For $n=1$, the class $\mathcal S$ is quite  well understood by means of the Loewner theory (see, {\sl e.g.} \cite{Pommerenke}, \cite{G-K} or the recent survey \cite{Abate}). In fact, every element $f$ of $\mathcal S$ for $n=1$ can be embedded  into a normalized Loewner chain whose range is $\C$ (see \cite[Thm. 6.1]{Pommerenke},  which by itself can be described by a differential equation with a particular driving term.  Thus
 $\mathcal S=\mathcal S^1$ and such a class is also compact (see \cite{Pommerenke}).

In higher dimension, in \cite{K}, \cite{GKH}, I. Graham, H. Hamada and G. Kohr introduced and studied the class $\mathcal S^0$ which is the subclass of $\mathcal S^1$ formed by those $f\in \mathcal S$ which admit a {\sl parametric representation}, namely, such that $f$ can be embedded into a normalized Loewner chain $(f_t)$ with the property that $\{e^{-t} f_t(\cdot)\}_{t\geq 0}$ is a normal family. A geometric characterization of maps in the class $\mathcal S^0$ is in \cite{GKH2}. The class $\mathcal S^0$ is compact and in dimension one it holds $\mathcal S^0=\mathcal S^1=\mathcal S$, while in higher dimension   one has $\mathcal S^0\subsetneq\mathcal S^1$ (see \cite[Section 8]{G-K}).

It is then natural to ask whether $\mathcal{S}=\mathcal{S}^1$  for $n>1$. It turns out that this is not always the case:  since every normalized Loewner chain is a $L^\infty$ Loewner chain in the sense of \cite{Arosio} (see Proposition \ref{image}), as explained in \cite[Section 4]{ArosioBracciWold} it follows from a theorem of Docquier-Grauert \cite{Docquier-Grauert}
that if $f\in \mathcal S^1$ then $f(\mathbb B^n)$ is a Runge domain. Therefore, if $f\in\mathcal S$ is such that
$f(\mathbb B^n)$ is not Runge then $f\not\in\mathcal S^1$ (and normalized univalent maps of the ball with non-Runge image do exist, see Example \ref{example}).  There is, however, the possibility that $f$ embeds into
a Loewner chain whose range is a  non-Runge Fatou-Bieberbach domain in $\mathbb C^n$.   Thus the two following natural questions arise:
\begin{itemize}
\item[Q1)]  Does any $f\in\mathcal S$ with $f(\mathbb B^n)$ Runge embed into a normalized Loewner chain $(f_t)$ with $R(f_t)=\C^n$?
\item[Q2)]  Does any $f\in\mathcal S$ embed into a normalized Loewner chain whose Loewner range is biholomorphic to $\mathbb C^n$?
\end{itemize}

In this note we discuss the two questions and give a partial positive answer to question Q1. Let us denote by $\mathcal S_R:=\{f\in \mathcal S: f(\B^n) \ \hbox{is Runge}\}$. Question Q1 can then be  restated in the following way: is $\mathcal{S}^1=\mathcal{S}_R?$
Using  Anders\'en-Lempert approximation theorem \cite[Thm. 2.1]{AndersenLempert} it is proved in \cite[Theorem 2.3]{S} that
\begin{equation}\label{close}
\overline{\mathcal S^1}=\mathcal S_R,
\end{equation}
where the closure has to be understood in terms of the topology of uniform convergence on compacta of $\B^n$.
However, one can say more: each $f\in\mathcal S_R$ can  be approximated by maps in $\mathcal S$ which can be embedded into ``nice'' normalized Loewner chains whose range is $\C^n$.

\begin{definition} Let $(f_t)$ be a normalized Loewner chain in $\B^n$. We say that $(f_t)$ is a {\sl filtering} normalized Loewner chain provided  the family $(\Omega_t)_{t>s}$ is a neighborhood basis for $\overline{\Omega}_s$ for all $s\geq 0$, {\it i.e.}, the following conditions hold:
\begin{itemize}
\item[(M1)] $\overline{\Omega}_s\subset \Omega_t$ for all $t>s$ and
\item[(M2)]       for any open set $U$ containing $\overline{\Omega}_s$ there exists $t_0>s$ such that $\Omega_t\subset U$  for all $t\in (s,t_0)$.
  \end{itemize}
\end{definition}

As a matter of notation, we let
\[
\mathcal S^1_{\mathfrak{F}}:=\{f\in \mathcal S: f \ \hbox{embeds in a filtering normalized Loewner chain with } R(f_t)=\C^n\}.
\]

Clearly, $\mathcal S^1_{\mathfrak{F}}\subset \mathcal S^1\subset \mathcal S_R$, and
$\overline{\mathcal S^1_{\mathfrak{F}}}=\mathcal S_R$ (see Corollary \ref{schleiss}). The main result of this note is the following:

\begin{theorem}\label{main2}
Let $n\geq 2$. Let $f\in \mathcal S_R$. Assume that $\Omega:=f(\B^n)$ is a bounded strongly pseudoconvex domain with $\mathcal C^\infty$
boundary. Then $f\in \mathcal S^1_{\mathfrak{F}}$ if and only if $\overline{\Omega}$ is  polynomially convex.
\end{theorem}

The proof of Theorem \ref{main2} is the content of Section \ref{proof1}.
Finally, in Section \ref{examples} we make some final remarks and discuss the role of polynomial convexity in the embedding problem, constructing various examples which explain the r\^{o}le of the hypotheses.

\medskip

We thank the referee for helpful comments which improved the original manuscript.

\section{Loewner chains and Loewner range}\label{catene}

A general Loewner theory in the unit disc and hyperbolic manifolds was introduced in \cite{B1, B2, C, Arosio} (see also \cite{AB}). According to \cite{Arosio},  a \emph{Loewner chain  of order $d\in [1,\infty]$} in $\mathbb C^n$, is a family of univalent mappings ($f_t:\mathbb B^n\rightarrow\mathbb C^n)_{t\in [0,\infty)}$, with   $\Omega_s:=f_s(\mathbb B^n)\subset f_t(\mathbb B^n)$ for $s\leq t$, such that for any compact set $K\subset \mathbb B^n$ it holds
\begin{equation}\label{condition1}
|f_s(z)-f_t(z)|\leq \int_{s}^t k_K(\zeta)d\zeta, \mbox{ for some } k_K\in L^d_{\sf loc}([0,+\infty),\R^+), \forall z\in K.
\end{equation}

The  set $R(f_t):=\cup_{s\geq 0}\Omega_s$ is a domain in $\C^n$, whose class of biholomorphism is called the {\sl Loewner range} of $(f_t)$. The Loewner range of a Loewner chain is an invariant for the associated Loewner equation (see \cite{Arosio}, \cite{ArosioBracciWold}).

The normalization required in the definition of a normalized Loewner chain, forces such a family to be a Loewner chain of order $\infty$ with range biholomorphic to $\C^n$:

\begin{proposition}\label{image}
Let $(f_t)$ be  a normalized Loewner chain. Then $(f_t)$ is a Loewner chain of order $\infty$  and its Loewner range is biholomorphic to $\C^n$.

Moreover, if the family $\{e^{-t}\circ f_t\}_{t\geq 0}$ is uniformly bounded in a neighborhood of the origin, then $R(f_t)=\C^n$.
\end{proposition}
\begin{proof}
According to \cite[Theorem 8.1.8]{G-K}, the family $(f_t)$ is locally Lipschitz continuous in $t$ locally uniformly with respect to $z$, hence it satisfies \eqref{condition1} with $k_K$ a positive constant. Therefore $(f_t)$ is a Loewner chain of order $\infty$.

The family $(\v_{s,t}\colon \B^n \to \B^n)_{0\leq s\leq t}$ defined as $\v_{s,t}:=f_t^{-1}\circ f_s$ for all $0\leq s\leq t$ is an  evolution family (of order $\infty$ according to \cite{Arosio}) which satisfies
 $$\v_{s,t}(z)=e^{s-t}z+O(|z|^2),\quad 0\leq s\leq t.$$
By \cite[Theorem 8.1.5]{G-K} there exists a normalized Loewner chain $(g_t\colon \B^n \to \C^n)$ associated with $(\v_{s,t})$ such that  $R(g_t)=\C^n$. But then $R(f_t)$ is biholomorphic to $\C^n$ by \cite[Corollary 4.8]{Arosio}.

Let $h_t:=e^{-t}\circ f_t$, $t\geq 0$. If $\{h_t\}_{t\geq 0}$ is uniformly bounded in a neighborhood of the origin, then there exists a ball $s\B\subset \bigcap_{t\in \R^+} h_t(\B^n)$, and thus $$\C^n=\bigcup_{t\in \R^+}e^{t}(s\B)\subset \bigcup_{t\in \R^+}e^{t}(h_t(\B^n)),$$
as claimed.
\end{proof}
In \cite[Section 4]{ArosioBracciWold} it has been shown that if $(f_t)$ is a Loewner chain of order $d$, then  $(\Omega_s,\Omega_t)$ is a Runge pair for all $0\leq s\leq t$, and, as a consequence, $\Omega_t$ is Runge in $R(f_t)$ for all $t\geq 0$. If $(f_t)$ is a normalized Loewner chain, then $R(f_t)$ is biholomorphic to $\C^n$, but it might happen that $\Omega_t$ is not Runge in $\C^n$:

\begin{example}\label{example}
Let $D$ be a Fatou-Bieberbach domain of $\C^2$ which is not Runge (see \cite{Erlend}), assume $0\in D$ and let $\phi:\C^n \to D$ be a biholomorphism such that $\phi(0)=0$ and  $d\phi_0={\sf id}$. Then $f_t(z):=\phi(e^t z)$, $t\geq 0$, $z\in \B^n$ is a normalized Loewner chain. Note that $R(f_t)=D$ is not Runge in $\C^n$, hence $\Omega_t$ is not Runge in $\C^n$ for $t$ sufficiently big.
\end{example}

\section{Embedding into filtering Loewner chains}\label{proof1}

We first make the following simple observations.
\begin{remark}\label{converse}
If  $f\in \mathcal S_{\mathfrak{F}}^1$, then  $\overline{f(\B^n)}$ is polynomially convex, since it has a Stein and Runge neighborhood basis given by $(\Omega_t)_{t>0}$.
\end{remark}
A domain $D\subset\C^n$ is called {\sl convexshapelike} if there exists an automorphism $\psi\in  Aut_{hol}\mathbb C^n$ such that $\psi(D)$ is convex in $\C^n$.
\begin{proposition}\label{trivial}
Let $f\in\mathcal S_R$, and let   $\Omega:=f(\B^n)$. If
 $\Omega$ is a bounded  convexshapelike domain then $f\in \mathcal S_{\mathfrak{F}}^1$.
\end{proposition}
\begin{proof}
Assume $\Omega$ is a convexshapelike domain.  Let $\psi\in Aut_{hol}\mathbb C^n$ be such that $\psi(0)=0$ and $\psi(\Omega)$ is convex. Hence the family of univalent mappings $f_t(z):=\psi^{-1}(e^t\cdot \psi(f(z)))$ has the property  $f_s(\B^n)\subset f_t(\B^n)$ for all $0\leq s\leq t$, and it is easy to check that it forms a filtering normalized Loewner chain such that $f_0=f$.
\end{proof}

\begin{corollary}\label{schleiss}
$\overline{\mathcal S^1_{\mathfrak{F}}}=\mathcal S_R.$
\end{corollary}
\begin{proof}
(cf. also \cite[Thm 2.3]{S}) Let $f\in \mathcal S_R$.
By the And\'ersen-Lempert theorem \cite[Thm. 2.1]{AndersenLempert} there exists a sequence $\{\psi_m\}_{m\in \N}$ of automorphisms of $\C^n$ such that $\psi_m\to f$ uniformly on compacta of $\B^n$. We can assume that $\psi_m(0)=0$ and $d(f_m)_0=\id$ for all $m\geq 0$. Since $\psi_m(\B^n)$ is by construction a bounded  convexshapelike domain, we have that ${\psi_m}|_{\B^n}\in  \mathcal S_{\mathfrak{F}}^1$.
\end{proof}

\begin{remark}\label{polcon}
Let $\Omega\subset\subset \C^n$ be a convexshapelike domain. Then $\Omega$ is Runge and $\overline{\Omega}$ is polynomially convex. Indeed, it is clear that $\Omega$ is Runge because convex domains are Runge and Runge-ness is invariant under automorphisms of $\C^n$. Moreover,  if $\psi\in  Aut_{hol}\mathbb C^n$ is such that $0\in \psi(\Omega)$ and $\psi(\Omega)$ is convex, then $\{\psi^{-1}(t\psi(\Omega))\}_{t>1}$ is a Runge and Stein neighborhood basis of $\overline\Omega$, which is thus polynomially convex (\cite[Thm. 1.11]{Range}).
\end{remark}

The proof of Theorem \ref{main2} depends on  the following well known Mergelyan type result, for which
we give a proof for a lack of a suitable reference.

\begin{lemma}\label{Mergelyan}
Let $\Omega\subset\mathbb C^n$ be a bounded strongly pseudoconvex domain with $\mathcal C^\infty$ boundary
which is biholomorphic to $\mathbb B^n$.  Then any $f\in\mathcal C^2(\overline\Omega)\cap \mathcal{O}(\Omega)$
can be approximated uniformly on $\overline\Omega$ in $\mathcal C^2$-norm, by
functions in $\mathcal O(\overline\Omega)$.
\end{lemma}
\begin{proof}
Let $X$ denote the radial vector field $X(z)=-\sum_{j=1}^nz_j\frac{\partial}{\partial z_j}$ defined on $\overline{\mathbb{B}^n}.$
By Fefferman's theorem \cite{Fefferman} the biholomorphism between the ball and $\Omega$ extends smoothly to the boundary and we can push $X$ forward to $\overline{\Omega}$, in order to get a vector field $\tilde{X}$ pointing into $\Omega$ on $\partial\Omega$.  By  \cite[Thm. 2.1 p. 280]{Range},
$\tilde{X}$ may be approximated uniformly on $\overline\Omega$ by vector fields in $\mathcal O(\overline\Omega)$.
In particular, there exists  a vector field $Y\in\mathcal O(\overline\Omega)$ pointing into $\Omega$ on $\partial\Omega$. Hence its real flow $\varphi_t$ is well defined for all $t\geq 0$ and $\varphi_t(\overline\Omega)\subset \Omega$ for $t>0$.

Let $f\in\mathcal C^2(\overline\Omega)\cap \mathcal{O}(\Omega)$. Then $f\circ\varphi_t\in \mathcal O(\overline\Omega)$ for all $t>0$, and $f\circ\varphi_t\rightarrow f$ in $\mathcal C^2$-norm
uniformly on $\overline\Omega$ as $t\rightarrow 0^+$.
\end{proof}

The proof of our result relies also on the following result, which might be interesting on its own.

\begin{proposition}\label{strongly-starlike}
Let $n\geq 2$.
Let $\Omega\subset \C^n$ be a bounded strongly pseudoconvex domain with $\mathcal C^\infty$ boundary which is biholomorphic to $\B^n$. Then the following  are equivalent:
\begin{enumerate}
  \item $\overline\Omega$ is  polynomially convex,
  \item  $\Omega$ is convexshapelike.
\end{enumerate}
Moreover, if condition (1) or (2) holds, then $\Omega$ is Runge.
\end{proposition}

\begin{proof}
$(2)\Rightarrow (1)$ it follows from Remark \ref{polcon} (and also implies that $\Omega$ is Runge).

$(1)\Rightarrow (2)$. Assume that $\overline\Omega$ is polynomially convex.  The aim is to find $\psi\in Aut_{hol}\mathbb C^n$ such that $\psi(\Omega)$ is a  convex domain (in fact a strongly convex domain with smooth boundary).

In order to find such an automorphism $\psi$, we are going to find an open neighborhood $U$ of $\overline\Omega$ and an univalent map $h:U \to \C^n$ such that
\begin{enumerate}
  \item $U$ is Runge,
  \item $h(U)$ is  starlike (in fact, convex),
  \item $h(\Omega)$ is a  strongly convex domain with smooth boundary.
\end{enumerate}
Once we have that, by the Anders\'{e}n-Lempert theorem \cite[Thm. 2.1]{AndersenLempert}, we can approximate $h^{-1}$ on $h(U)$---and hence $h$ on $U$---uniformly on compacta with automorphisms of $\C^n$. Hence we can find an automorphism $\psi$ having the required properties, and we are done.

In order to construct $h$, let $f:\B^n\to \Omega$ be a biholomorphism. We note that by Fefferman's theorem \cite{Fefferman}
$f$ extends to a diffeomorphism $f:\overline{\mathbb B^n}\rightarrow\overline\Omega$.
By Lemma \ref{Mergelyan},  $f^{-1}$ can be approximated in $\mathcal C^2$-norm uniformly on $\overline\Omega$ by
holomorphic maps defined on neighborhoods of $\overline\Omega$. Therefore, there exists an open neighborhood $U'$ of $\overline\Omega$ and a univalent map $h: U' \to \C^n$ such that  $h(\Omega)$ is a smooth strongly convex domain. Since a compact polynomially convex set admits a basis of Stein neighborhoods that are Runge in $\C^n$, and by hypothesis $\overline\Omega$ is polynomially convex, we can find an open set $U''$ such that $U''$ is  Runge and $\overline{\Omega}\subset U''\subset U'$. Now, since $h(\overline{\Omega})$ has  a basis of convex neighborhoods, we can find a convex set $A$ such that $h(\overline{\Omega})\subset A\subset h(U'')$. Therefore, $U:=h^{-1}(A)$ is Runge in $U''$ and since the latter is Runge in $\C^n$, it follows that $U$ is Runge in $\C^n$, and the proof is concluded.
\end{proof}
\begin{remark}
Notice that Proposition \ref{strongly-starlike} is false in dimension one, because the group of automorphisms of $\C$ is too ``small''.
\end{remark}

The proof of  Theorem \ref{main2} is now straightforward:
\begin{proof}[Proof of Theorem \ref{main2}]
If  $\overline\Omega$ is polynomially convex, then by Proposition \ref{strongly-starlike} it follows that $\Omega$ is convexshapelike and hence by Proposition \ref{trivial} it follows $f\in \mathcal S^1_{\mathfrak{F}}$. The converse is  Remark \ref{converse}.
\end{proof}
\begin{remark}
As an application, we can give an alternative proof of Corollary \ref{schleiss}, more in the spirit of the one-dimensional proof by Ch. Pommerenke (\cite[Thm. 6.1]{Pommerenke}).
 Let thus $f\in \mathcal S_R$, and consider $f_r(z):=\frac{1}{r}f(rz)$, for $0<r<1$. Then clearly $f_r\in \mathcal S$, $f_r$ converges uniformly on compacta of $\B^n$ to $f$, and $f_r(\B^n)$ is a strongly pseudoconvex domain with smooth boundary. Since $\frac{1}{r}f(\ell\B^n)$ with $r<\ell\leq 1$ is a  basis of Stein neighborhoods of $\overline{f_r(\B^n)}$ that are Runge in $\C^n$ (because $f(\ell\B^n)$ is Runge in $f(\B^n)$ which by hypothesis is Runge in $\C^n$), it follows that $\overline{f_r(\B^n)}$ is polynomially convex. Hence $f_r\in \mathcal S^1_{\mathfrak{F}}$.
\end{remark}

\section{Remarks and Examples}\label{examples}

Let $f\in \mathcal S_R$. By Remark \ref{converse}, a necessary condition  for having $f\in \mathcal{S}^1_{\mathfrak{F}}$ is that
 $\overline{f(\B^n)}$ is polynomially convex. Such a condition, as in the one-dimensional case, is not necessary for having $f\in \mathcal{S}^1$, as the following example  shows.
\begin{example}\label{first-e}
Let $\phi:\D \to \C$ be a univalent map, $\phi(0)=0$, $\phi'(0)=1$. Then $f(z,w)=(\phi(z), w)\in \mathcal S$. Let  $(\phi_t)_{t\geq 0}$ be a normalized Loewner chain in $\D$ such that $\phi_0=\phi$. Define $f_t(z,w):=(\phi_t(z), e^t w)$ for $(z,w)\in\B^2$. Then $(f_t)$ is a normalized Loewner chain such that $f_0=f$. Indeed, setting $\v_{s,t}(z,w):=(\phi_t^{-1}\circ \phi_s(z), e^{s-t}w)$,   by the Schwarz lemma
\[
|\phi_t^{-1}\circ \phi_s(z)|^2+|e^{s-t}w|^2\leq |z|^2+|w|^2<1,
\]
from which it follows that $\v_{s,t}(\B^2)\subset \B^2$, and $f_t \circ \v_{s,t}=f_s$ for $0\leq s<t$, which implies that  $f_s(\B^2)\subset f_t(\B^2)$ for $0\leq s<t$.

Now, recall that a compact set $K\subset \C$ is polynomially convex if and only if $\C\setminus K$ is simply connected. Let $D\subset \C$ be a simply connected domain such that $\C\setminus \overline{D}$ is not simply connected (for instance take $D=\{\zeta\in \D: |\zeta-1/2|>1/2\}$). Let $\phi:\D \to D$ be a Riemann mapping. Up to rescaling, we can assume $\phi(0)=0$ and $\phi'(0)=1$. Let $f:\B^2\to \B^2$ be defined by $f(z,w):=(\phi(z), w)$. Then $f\in \mathcal S^1$ but $\overline{f(\B^2)}$ is not polynomially convex.
\end{example}

Next we construct an example of a map $f\in \mathcal S_R$ such that $f$ extends holomorphically through $\de \B^2$ but $\overline{f(\B^2)}$ is not polynomially convex and hence $f\not\in \mathcal S^1_{\mathfrak{F}}$.

\begin{example}
In order to construct such an $f$, let  $\varphi:\mathbb B^2\rightarrow\mathbb C^2$ be a univalent mapping such that $\varphi(\mathbb B^2)$ is not a Runge domain (the existence of $\varphi$ follows from  the existence of a non-Runge Fatou-Bieberbach domain in $\mathbb C^2$ \cite{Erlend}). Assume $\varphi(0)=0$ and $d\varphi_0=\id$.

 Let
$$
\vartheta:=\sup\{t:\overline{\varphi({r\mathbb B^2})} \mbox{ is polynomially convex for all } r<t\}.
$$
Then, since  $\overline{\varphi({r \mathbb B^2})}$ is polynomially convex   for $r$ small enough, one has $\vartheta>0$. Moreover, $\vartheta<1$ since  $\varphi(\mathbb B^2)$ is not  Runge (and thus it cannot admit any growing exhaustion by polynomially convex sets). We claim that $K:=\overline{\varphi(\vartheta\mathbb B^2)}$ is not polynomially convex. Assume it is, by contradiction. Then $K$ admits a Runge neighborhood basis, and thus there exists a Runge domain $\Omega$ such that $$K\subset \Omega\subset \varphi (\B^2).$$ Let $\vartheta<u<1$ be such that $\overline{\varphi(u\mathbb B^2)}\subset\Omega$.  One has that $\overline{u\mathbb B^2}$ is holomorphically convex in $\B^2$, and thus $\overline{\varphi(u\mathbb B^2)}$ is holomorphically convex in $\Omega$, hence it is polynomially convex, contradicting the assumption that $\vartheta$ was the supremum.

Finally, define $f(z):= \frac{1}{\vartheta}\varphi(\vartheta z)$. Since  $\overline{f(\mathbb B^2)}$ is not  polynomially convex, it follows from Remark \ref{converse} that $f\not\in\mathcal S^1_{\mathfrak{F}}.$

We do not know whether $f\in \mathcal{S}^1$.
\end{example}

There are also examples of maps $f\in \mathcal S^1$ such that $\overline{f(\B^n)}$ is polynomially convex but $f\not\in\mathcal S^1_{\mathfrak{F}}$ (and $f(\B^n)$ has not smooth boundary). A simple example in dimension one is given by  a Riemann mapping $f$ from the unit disc $\D$ onto a disc $D$ minus a slit. If such a function were embeddable into a filtering Loewner chain $(f_t)$, then for all $t>0$ the closed disc $\overline{D}$ would be contained in $f_t(\D)$, hence, by the Carath\'eodory kernel convergence theorem, $\overline{D}\subset f_0(\D)$, a contradiction.

We end up this note with the following consideration about Question Q2:

\begin{remark}
Let $f\in \mathcal S$. Assume that $f(\B^n)$ is not Runge but suppose there exists a Fatou-Bieberbach $D\subset \C^n$ such that $0\in f(\B^n)\subset D$ and $f(\B^n)$ is Runge in $D$. Then let $F:D\to \C^n$ be a biholomorphism such that $F(0)=0, dF_0={\sf id}$. It follows that $F\circ f\in \mathcal S_R$. Thus, if $(g_t)$ is a normalized Loewner chain  with range $\C^n$ such that $g_0=F\circ f$, then $(f_t:=F^{-1} \circ g_t)$ is a normalized Loewner chain with $R(f_t)=D$ such that $f_0=f$.
\end{remark}

\end{document}